\documentclass[12pt,reqno]{amsart}
\usepackage[utf8]{inputenc}
\usepackage{lmodern}
\usepackage{amsmath}
\usepackage{amssymb}
\usepackage{amsthm}
\usepackage{mathrsfs}
\usepackage{enumerate}
\usepackage{stmaryrd,newlfont}

\usepackage{graphicx}
\usepackage{color}
\usepackage{tikz}
\usetikzlibrary{calc,shapes.arrows,decorations.pathreplacing}

\newcommand{\N}{\mathbb{N}}
\newcommand{\Z}{\mathbb{Z}}

\newcommand{\R}{\mathbb{R}}
\newcommand{\C}{\mathbb{C}}

\newcommand{\T}{\mathbb{T}}
\newcommand{\DT}{\mathcal{D}'(\T)}
\newcommand{\MT}{\mathcal{M}(\T)}

\newcommand{\A}[1]{A^{#1}(\mathbb{T})}

\newcommand{\dd}{~\mathrm{d}}

\newcommand{\supp}{\mathrm{supp}}

\newcommand{\PT}{\mathcal{P}(\T)}

\newcommand{\cZ}{{\mathcal{Z}}}
\DeclareMathOperator{\dist}{dist}

\title[Cyclicity in weighted $\ell^p$ spaces]{Cyclicity in weighted $\ell^p$ spaces}
\author[F. Le Manach]{Florian Le Manach}
\address{IMB\\Universit\'e de  Bordeaux \\
351 cours de la Lib\'eration\\33405 Talence \\France}
\email{florian.le-manach@math.u-bordeaux.fr}

\newtheorem{Def}{Definition}[section]
\newtheorem{lemme}[Def]{Lemma}
\newtheorem{prop}[Def]{Proposition}
\newtheorem{theo}[Def]{Theorem}
\newtheorem*{theoA}{Theorem A}
\newtheorem*{theoB}{Theorem B}

\numberwithin{equation}{section}
 
\begin{document}

\begin{abstract} 
We study the cyclicity in weighted $\ell^p(\Z)$ spaces.  For $p \geq 1$ and $\beta \geq 0$,  let 
$\ell^p_\beta(\Z)$ be the space of sequences $u=(u_n)_{n\in \Z}$ such that 
 $(u_n |n|^{\beta})\in \ell^p(\Z) $. We obtain both necessary conditions and sufficient conditions for $u$ to be cyclic in  $\ell^p_\beta(\Z)$,  in other
words, for $ \{(u_{n+k})_{n \in \Z},~ k \in \Z \}$  to span a dense subspace of $\ell^p_\beta(\Z)$.
 The conditions are given in terms of the Hausdorff dimension and the capacity of the zero set of the Fourier transform  of $u$.
\end{abstract}

\maketitle

\section{Introduction and main results}

For $p \geq 1$ and $\beta \in \R$, we define the Banach space 
$$\ell^p_\beta(\Z) = \left\{ u=(u_n)_{n\in \Z} \in \C^\Z, \;\; 
 \|u\|_{\ell^p_\beta}^p = \sum_{n\in \Z} |u_n|^p (1+|n|)^{p\beta} < \infty \right\}$$
endowed with the norm $\|\cdot\|_{\ell^p_\beta}$. Notice that $\ell^p_0(\Z)$ is the classical $\ell^p(\Z)$ space.

\bigskip
In this work, we are going to investigate cyclic vectors for $\ell^p_\beta(\Z)$ when $\beta \geq 0$. A vector $u \in \ell^p_\beta(\Z)$ is called {\it cyclic} in $\ell^p_\beta(\Z)$ if the linear span of $ \{(u_{n+k})_{n \in \Z},~ k \in \Z \}$ is dense in $\ell^p_\beta(\Z)$.

\bigskip
We denote by $\T$ the circle $\R / 2\pi \Z$. The Fourier transform of $u \in \ell^p(\Z)$ is given by
$$\widehat{u} : t \in \T \mapsto \sum_{n\in \Z}u_n e^{int}$$
and when $\widehat{u}$ is continuous, we denote by $\cZ(\widehat{u})$ the zero set on $\T$ of $\widehat{u}$:
$$\cZ(\widehat{u})=\{t \in \T,~ \widehat{u}(t)=0\}.$$

\bigskip
The case $\beta=0$ was already studied by Wiener, Beurling, Salem and Newman. 
When $p=1$ or $p=2$, Wiener characterized the cyclic vectors $u$ in $\ell^p(\Z)$ by the zeros of $\widehat{u}$, with the following theorem. 

\begin{theo}[\cite{WIE}] \label{WienerTh}
Let $u \in \ell^p(\Z)$.
\begin{enumerate}
\item If $p=1$ then $u$ is cyclic in $\ell^1(\Z)$ if and only if $\widehat{u}$ has no zeros on $\T$.
\item If $p=2$ then $u$ is cyclic in $\ell^2(\Z)$ if and only if $\widehat{u}$ is non-zero almost everywhere.
\end{enumerate}
\end{theo}

Lev and Olevskii showed that, for $1<p<2$ the problem of cyclicity in $\ell^p(\Z)$ is more complicated even for sequences in $\ell^1(\Z)$. 
The following Theorem  of Lev and Olevskii contradicts the Wiener conjecture.
\begin{theo}[\cite{LEV}]
If $1 < p < 2$, there exist $u$ and $v$ in $\ell^1(\Z)$ such that $\cZ(\widehat{u})=\cZ(\widehat{v})$, $u$ is not cyclic in $\ell^p(\Z)$, and $v$ is cyclic in $\ell^p(\Z)$.
\end{theo}
So we can't characterize the cyclicity of $u$ in $\ell^p(\Z)$ in terms of only $\cZ(\widehat{u})$, the zero set of $\widehat{u}$. However for $u \in \ell^1(\Z)$, Beurling, Salem and Newman gave both necessary conditions and sufficient conditions for $u$ to be cyclic in $\ell^p(\Z)$. These conditions rely on the "size" of the set $\cZ(\widehat{u})$ in term of it's $h$-measure, capacity and Hausdorff dimension.

\bigskip
Given $E \subset \T$ and $h$ a continuous function, non-decreasing and such that $h(0)=0$, we define the $h$-measure of E by
$$H_h(E) = \lim_{\delta \to 0} \inf \left\{ \sum_{i=0}^\infty h(|U_i|),~ E \subset \bigcup_{i=0}^\infty U_i,~ |U_i| \leq \delta \right\}$$
where the $U_i$ are open intervals of $\T$ and where $|U_i|$ denotes the length of $U_i$.\\
The Hausdorff dimension of a subset $E \subset \T$ is given by $$\dim(E) = \inf\{\alpha \in (0,1), H_\alpha(E)=0 \} = \sup\{\alpha \in (0,1), H_\alpha(E)=\infty \},$$ where $H_\alpha = H_h$ for $h(t)=t^\alpha$ (see \cite{KAH},  pp. 23-30).\\
Let $\mu$ be a positive measure on $\T$ and $\alpha \in [0,1)$. We define the $\alpha$-energy of $\mu$ by 
$$I_\alpha(\mu) = \sum_{n \geq 1} \frac{|\widehat{\mu}(n)|^2}{(1+|n|)^{1-\alpha}}.$$
The $\alpha$-capacity of a Borel set $E$ is given by
$$C_\alpha(E) = 1/{\inf \{ I_\alpha(\mu),~ \mu \in \mathcal{M}_\mathcal{P}(E) \} },$$
where $\mathcal{M}_\mathcal{P}(E)$ is the set of all probability measures on $\T$ which are supported on a compact subset of $E$. If $\alpha = 0$, $C_0$ is called the logarithmic capacity.\\
An important property which connects capacity and Hausdorff dimension is that (see \cite{KAH},  p. 34)
\small
\begin{equation} \label{PropCap}
\dim(E) = \inf\{\alpha \in (0,1), C_\alpha(E)=0 \} = \sup\{\alpha \in (0,1), C_\alpha(E)>0 \}.
\end{equation}
\normalsize

In the following theorem, we summarize the results of  Beurling \cite{BEU}, Salem  \cite{SAL} (see 	also \cite{KAH} pp. 106-110) and Newman  \cite{NEW}.   The H\"older conjugate of $p\neq 1$ is noted by $q = \frac{p}{p-1}$. 

\begin{theo}[\cite{BEU, NEW, SAL}]  \label{ThSansPoids}
Let $1 \leq p \leq 2$.
\begin{enumerate}
\item If $u \in \ell^1(\Z)$ and $\dim(\cZ(\widehat{u})) < {2}/{q}$ then $u$ is cyclic in $\ell^p(\Z)$.
\item For ${2}/{q} < \alpha \leq 1$, there exists $E \subset \T$ such that $\dim(E)=\alpha$ and every $u \in \ell^1(\Z)$ satisfying $\cZ(\widehat{u})=E$ is not cyclic in $\ell^p(\Z)$.
\item There exists $E \subset \T$ such that $\dim(E)=1$ and every $u \in \ell^1(\Z)$ satisfying $\cZ(\widehat{u})=E$ is cyclic in $\ell^p(\Z)$ for all $p>1$.
\end{enumerate}
\end{theo}

In this paper we give a generalization of the results of Beurling, Salem and Newman to $\ell^p_\beta(\Z)$ spaces.\\
When $\beta q > 1$, we have an analogue of $(1)$ in Wiener's Theorem \ref{WienerTh}: a vector $u \in \ell^p_\beta(\Z)$ is cyclic if and only if $\widehat{u}$ has no zeros on $\T$. Indeed, $\ell^p_\beta(\Z)$ is a Banach algebra if and only if $\beta q > 1$ (see \cite{ZAR}).\\
When $p=2$, Richter, Ross and Sundberg gave a complete characterization of the cyclic vectors $u$ in the weighted harmonic Dirichlet spaces $\ell^2_\beta(\Z)$ by showing the following result:

\begin{theo}[\cite{RRS}]
Let $0<\beta \leq \frac{1}{2}$ and $u \in \ell^1_\beta(\Z)$ .\\
The vector $u$ is cyclic in $\ell^2_\beta(\Z)$ if and only if $C_{1-2\beta}(\cZ(\widehat{u}))=0$.
\end{theo}

Our first main result is the following theorem.

\begin{theoA}
Let $1 < p < 2$, $\beta > 0$ such that $\beta q \leq 1$. 
\begin{enumerate}
\item If  $u \in \ell^1_\beta(\Z)$ and $\dim(\cZ(\widehat{u})) < \frac{2}{q}(1 - \beta q)$ then $u$ is cyclic in $\ell^p_\beta(\Z)$.
\item If  $u \in \ell^1_\beta(\Z)$ and $\dim(\cZ(\widehat{u})) > 1-\beta q$ then $u$ is not cyclic in $\ell^p_\beta(\Z)$.
\item For $\frac{2}{q}(1 - \beta q) < \alpha \leq 1$, there exists a closed subset $E \subset \T$ such that $\dim(E)=\alpha$ and every $u \in \ell^1_\beta(\Z)$ satisfying $\cZ(\widehat{u})=E$ is not cyclic in $\ell^p_\beta(\Z)$.
\item If $p = \frac{2k}{2k-1}$ for some $k \in \N^*$ there exists a closed subset $E \subset \T$ such that $\dim(E) = 1 - \beta q$ and every $u \in \ell^1_\beta(\Z)$ satisfying $\cZ(\widehat{u})=E$ is cyclic in  $\ell^p_\beta(\Z)$.
\end{enumerate}
\end{theoA}

Note that in order to prove $(2)$ and $(4)$ we show a stronger result (see Theorem \ref{ThAA}).

We can summarize Theorem A by the following diagram:

\bigskip

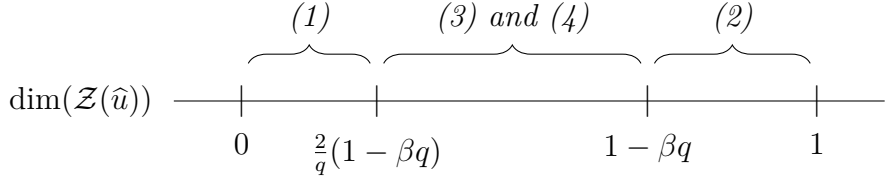
\begin{center}
\begin{tikzpicture}[scale=0.9]
\draw (0,0) -- (10.5,0);
\draw (0,0) node[left] {$\dim(\cZ(\widehat{u}))~~~$};
\draw (1,0) node[below,yshift=-0.3cm] {$0$} node {$|$};
\draw (3,0) node[below,yshift=-0.3cm] {$\frac{2}{q}(1 - \beta q)$} node {$|$};
\draw (7,0) node[below,yshift=-0.3cm] {$1-\beta q$} node {$|$};
\draw (9.5,0) node[below,yshift=-0.3cm] {$1$} node {$|$};
\draw[decorate,decoration={brace,raise=0.5cm,amplitude=7pt}] (1.1,0) -- ++(1.8,0) node[midway,above,yshift=0.75cm] {\textit{(1)}};
\draw[decorate,decoration={brace,raise=0.5cm,amplitude=7pt}] (3.1,0) -- ++(3.8,0) node[midway,above,yshift=0.75cm] {\textit{(3) and (4)}};
\draw[decorate,decoration={brace,raise=0.5cm,amplitude=7pt}] (7.1,0) -- ++(2.3,0) node[midway,above,yshift=0.75cm] {\textit{(2)}};
\end{tikzpicture}
\end{center}

\bigskip
\sloppypar{The fourth propriety shows that the bound $1 - \beta q$ obtained in $(2)$ is optimal in the sense that there is no cyclic vector such that $\dim(\cZ(\widehat{u})) > 1 - q \beta$, and, we can find some cyclic vector $u$ with $\dim(\cZ(\widehat{u})) = 1 - \beta q$. However this is only proved if $p=\frac{2k}{2k-1}$ for some positive integer $k$. When $p$ is not of this form, for all positive integer $k$,  we still prove similar results but we loose the optimality because we fail to reach the bound $1 - \beta q$.}

\bigskip
The "equality case" $\dim(\cZ(\widehat{u})) = \frac{2}{q}(1 - \beta q)$ is not treated by the previous theorem. Newman gave a partial answer to this question when $\beta=0$, by showing that, under some additional conditions on $\cZ(\widehat{u})$, $\dim(\cZ(\widehat{u})) = \frac{2}{q}$ implies that $u$ is a cyclic vector (see \cite[Theorem 1]{NEW}). We need the notion of strong $\alpha$-measure, $\alpha \in (0,1)$, to state Newman's Theorem in the equality case. For $E$ a compact subset of $\T$, we note $(a_k,b_k)$, $k \in \N$ its complementary intervals arranged in non-increasing order of lengths and set
\begin{equation} \label{defrn}
r_n = 2\pi - \sum_{k=0}^n (b_k-a_k).
\end{equation}
We will say that $E$ has strong $\alpha$-measure $0$ if $$\lim_{n \to \infty} ~ r_n ~ n^{\frac{1}{\alpha}-1} = 0.$$
Notice that if $E$ has strong $\alpha$-measure $0$ then $H_\alpha(E)=0$. The converse is true for some particular sets like Cantor sets but in general the converse is false (for some countable sets).

\begin{theo}
Let $1<p<2$ and $u \in \ell^1(\Z)$.\\
If $\cZ(\widehat{u})$ has strong $\alpha$-measure $0$ where $\alpha = \frac{2}{q}$ then $u$ is cyclic in $\ell^p(\Z)$.
\end{theo}

Moreover, in \cite{NEW}, Newman asked the question :
$$\textit{ For }u\in  \ell^1(\Z)\textit{,  does  }H_{2/q}(\cZ(\widehat{u}))=0\textit{ imply that } u \textit{ is cyclic in } \ell^p(\Z) \textit{ ?}$$
A positive answer to this question would contain Theorem \ref{WienerTh} and Theorem \ref{ThSansPoids}.$(1)$. We are not able to  answer this question completely. Nevertheless,  we show that if we replace $2/q$-measure by $h$-measure where $h(t)={t^{{2}/{q}}}{\ln(1/t)^{-\gamma}}$ with   $\gamma > \frac{2}{q}$ then the answer is negative.
Moreover we extend Newman's Theorem to $\ell^p_\beta(\Z)$.

\begin{theoB}\label{theoremeB}
Let $1 < p < 2$, $\beta \geq 0$ such that $\beta q < 1$.
\begin{enumerate}
\item If  $u \in \ell^1_\beta(\Z)$ and $\cZ(\widehat{u})$ has strong $\alpha$-measure $0$ where $\alpha = \frac{2}{q}(1 - \beta q)$ then $u$ is cyclic in $\ell^p_\beta(\Z)$.
\item For every $\gamma > \frac{2}{q}$, there exists a closed subset $E \subset \T$ such that every $u \in \ell^1_\beta(\Z)$ satisfying $\cZ(\widehat{u})=E$ is not cyclic in $\ell^p_\beta(\Z)$ and such that $H_h(E)=0$ where $h(t)={t^\alpha}{\ln(e/t)^{-\gamma}}$ with $\alpha = \frac{2}{q}(1 - \beta q)$ 
\end{enumerate}
\end{theoB}

Note that the set $E$ constructed in part (2) of Theorem B satisfy  $\dim(E) = \frac{2}{q}(1-\beta q)$.

\section{Preliminaries and lemmas}

Let $1 \leq p < \infty$ and $\beta \in \R$. 
We denote by $\DT$ the set of distributions on $\T$ and $\MT$ the set of measures on $\T$. For $S \in \DT$, we denote by $\widehat{S} = (\widehat{S}(n))_{n \in \Z}$ the sequence of Fourier coefficients of $S$ and we write $S = \sum_{n} \widehat{S}(n) e_n$, where $e_n(t)=e^{int}$. The space $A^p_\beta(\T)$ will be the set of all distributions $S \in \DT$ such that $\widehat{S}$ belongs to $\ell^p_\beta(\Z)$. We endow $A^p_\beta(\T)$ with the norm $\|S\|_{A^p_\beta(\T)} = \|\widehat{S}\|_{\ell^p_\beta}$. We will write $A^p(\T)$ for the space $A^p_0(\T)$. Thus the Fourier transformation is an isometric isomorphism between  $\ell^p_\beta(\Z)$ and $A^p_\beta(\T)$. We prefer to work with $A^p_\beta(\T)$ rather than $\ell^p_\beta(\Z)$. In this section we establish some properties of $A^p_\beta(\T)$ which will be needed to prove Theorems A and B.

\bigskip
For $1 \leq p < \infty$ and $\beta \geq 0$ we define the product of $f \in A^1_\beta(\T)$ and $S \in A^p_\beta(\T)$ by
$$fS = \sum_{n \in \Z} (\widehat{f} \ast \widehat{S}) (n) ~ e_n = \sum_{n \in \Z} \left( \sum_{k\in \Z} \widehat{f}(k)\widehat{S}(n-k) \right) e_n,$$
and we see that $\|fS\|_{A^p_\beta(\T)} \leq \|f\|_{A^1_\beta(\T)} \|S\|_{A^p_\beta(\T)}$. 
Note that if $S \in A^p_{-\beta}(\T)$ we can also define the product $fS \in A^p_{-\beta}(\T)$ by the same formula and obtain a similar inequality: $\|fS\|_{A^p_{-\beta}(\T)} \leq \|f\|_{A^1_\beta(\T)} \|S\|_{A^p_{-\beta}(\T)}$.

\bigskip
For $p \neq 1$, the dual space of $A^p_\beta(\T)$ can be identified with $A^q_{-\beta}(\T)$ ($q=\frac{p}{p-1}$) by the following formula
$$\langle S,T \rangle = \sum_{n \in \Z} \widehat{S}(n) \widehat{T}(-n), \qquad S \in A^p_\beta(\T),~ T \in A^q_{-\beta}(\T).$$

\bigskip
We denote by $\PT$ the set of trigonometric polynomials on $\T$. We rewrite the definition of cyclicity in the spaces $A^p_\beta(\T)$ for $\beta\geq 0$ : $S \in A^p_\beta(\T)$ will be a cyclic vector if the set $\{P S,~ P \in \PT \}$ is dense in $A^p_\beta(\T)$. It's clear that the cyclicity of  $S$ in $A^p_\beta(\T)$ is equivalent to the cyclicity of the sequence $\widehat{S}$ in $\ell^p_\beta(\Z)$.  Moreover 
for $1 \leq p < \infty$ and $\beta \geq 0$,  $S$  is cyclic in $A^p_\beta(\T)$ if and only if there exists a sequence $(P_n)$ of trigonometric polynomials such that
\begin{equation}\label{caraCyclNorm}
\lim_{n \to \infty} \|1-P_nS\|_{A^p_\beta(\T)} = 0.
\end{equation}

We need the following lemmas which gives us different inclusions between the $A^p_\beta(\T)$ spaces.

\begin{lemme} \label{inclusionAp}
Let $1 \leq r,s < \infty$ and $\beta,\gamma \in \R$.
\begin{enumerate}
\item If $r \leq s$ then $A^r_\beta(\T) \subset A^s_\gamma(\T) \Leftrightarrow \gamma \leq \beta$.
\item If $r>s$ then $A^r_\beta(\T) \subset A^s_\gamma(\T) \Leftrightarrow \beta - \gamma > \frac{1}{s} - \frac{1}{r}$.
\end{enumerate}
\end{lemme}

\begin{proof}
$(1)$ : We suppose that $r \leq s$. If $\gamma \leq \beta$ and $S \in A^r_\beta(\T)$, we have
$$\sum_{n \in \Z} |\widehat{S}(n)|^s (1+|n|)^{\gamma s} \leq \sum_{n \in \Z} |\widehat{S}(n)|^s (1+|n|)^{\beta s}.$$
Since $\| \cdot \|_{\ell^s} \leq \| \cdot \|_{\ell^r}$, we obtain $S \in A^s_\gamma(\T)$ and so $A^r_\beta(\T) \subset A^s_\gamma(\T)$.\\
Now suppose $\gamma > \beta$. Let $S \in \DT$ be given by
$$\widehat{S}(n)  (1+|n|)^\beta = \left\{
    \begin{array}{ll}
        {(1+m)^{-2/r}} & \mbox{if } |n|=2^m \\
        0 & \mbox{otherwise.}
    \end{array}
\right.$$
Then we have $S \in A^r_\beta(\T) \setminus A^s_\gamma(\T)$.

\bigskip
$(2)$ : Now suppose that $r>s$.
If $\beta - \gamma > \frac{1}{s} - \frac{1}{r}$, we have by H\"older's inequality,
$$\|S\|_{A^s_\gamma(\T)} \leq \|S\|_{A^r_\beta(\T)} \left( \sum_{n \in \Z} (1+|n|)^{\frac{rs}{r-s}(\gamma-\beta)} \right)^{1-s/r}, \qquad S \in A^r_\beta(\T),$$
so that $A^r_\beta(\T) \subset A^s_\gamma(\T)$.\\
Now suppose that $\beta - \gamma < \frac{1}{s} - \frac{1}{r}$. Let $\varepsilon > 0$ such that $\beta - \gamma + \varepsilon < \frac{1}{s} - \frac{1}{r}$, $\alpha = -\frac{1}{s} - \gamma + \varepsilon$ and let $S \in \DT$ be such that $\widehat{S}(n) = n^\alpha$. We have $S \in A^r_\beta(\T) \setminus A^s_\gamma(\T)$.

For the case $\beta - \gamma = \frac{1}{s} - \frac{1}{r}$ we take $S \in \DT$ such that $$\widehat{S}(n)^r (1+|n|)^{\beta r} = \frac{1}{(1+|n|) \ln (1+|n|)^{1+\varepsilon}}$$ with $\varepsilon = \frac{r}{s}-1 > 0$. We can show that $S \in A^r_\beta(\T) \setminus A^s_\gamma(\T)$ which proves that $A^r_\beta(\T) \not \subset A^s_\gamma(\T)$.
\end{proof}

\bigskip
For $E \subset \T$, we denote by $A^p_\beta(E)$ the set of $S \in A^p_\beta(\T)$ such that $\supp(S) \subset E$, where $\supp(S)$ denotes the support of the distribution $S$. The following lemma is a direct consequence of the definition of capacity (see \cite{KAH}) and the inclusion $A^q_{-\beta}(\T) \subset A^2_{\frac{\alpha-1}{2}}(\T)$ when $q \geq 2$ and $0 \leq \alpha < \frac{2}{q}(1-\beta q)$.

\begin{lemme} \label{KahaneSupp}
Let $E$ a Borel set, $\beta \geq 0$ and $q \geq 2$. If there exists $\alpha$, $0 \leq \alpha < \frac{2}{q}(1-\beta q)$, such that $C_\alpha(E)=0$ then 
$A^q_{-\beta}(E) = \{ 0 \}$.
\end{lemme}

\bigskip
We obtain the first results about cyclicity for the spaces $A^p_\beta(\T)$, when $A^p_\beta(\T)$ is a Banach algebra. More precisely, we have (see \cite{ZAR})

\begin{prop}
Let $1 \leq p < \infty$ and $\beta \geq 0$. 
$A^p_\beta(\T)$ is a Banach algebra if and only if $\beta q > 1$. Moreover when $\beta q > 1$, a vector $f \in A^p_\beta(\T)$ is cyclic in $A^p_\beta(\T)$ if and only if $f$ has no zeros on $\T$.
\end{prop}

\bigskip
Let $f \in A^1_\beta(\T)$ ans $S \in \DT$. We denote by $\cZ(f)$ the zero set of the function $f$. Recall that $e_n : t \mapsto e^{int}$.
\begin{lemme} \label{supportZeros}
Let $1 \leq p < \infty$ and $\beta \geq 0$. Let $f \in A^1_\beta(\T)$ and $S \in A^p_{-\beta}(\T)$. 
If for all $n \in \Z$, $\langle S,e_n f \rangle=0~$ then $~\supp(S) \subset \cZ(f)$.
\end{lemme}

\begin{proof}
We have
$$\langle S,e_n f \rangle= \langle fS,e_n \rangle = 0.$$

Hence $fS=0$. Let $\varphi \in C^\infty(\T)$ such that $\supp(\varphi) \subset \T \setminus \cZ(f)$. We claim that $\displaystyle \frac{\varphi}{f} \in A^1_\beta(\T) \subset A^q_\beta(\T)$ where $q=\frac{p}{p-1}$. So we obtain
$$\langle S,\varphi \rangle = \langle fS,\frac{\varphi}{f} \rangle =0$$
which proves that $\supp(S) \subset \cZ(f)$.\\
Now we prove the claim. Let $\varepsilon = \min\{|f(t)|,~ t \in \supp(\varphi) \} > 0$ and $P \in \PT$ such that $\|f-P\|_{A^1_\beta(\T)} \leq \varepsilon/3$.

By the Cauchy-Schwarz and Parseval inequalities,  for  every $g \in C^1(\T)$, we get 
 \begin{equation} \label{ineA1inf}
\|g\|_{A^1_\beta(\T)} \leq \|g\|_\infty + 2 \sqrt{\frac{2-2\beta}{1-2\beta}} \|g'\|_\infty.
\end{equation}

Now,  as in \cite{Npams}, by applying \eqref{ineA1inf} to $\displaystyle \frac{\varphi}{P^n}$ we see that 
$$\frac{\varphi}{f} = \sum_{n \geq 1} \varphi ~ \frac{(P-f)^{n-1}}{P^n} \in A^1_\beta(\T),$$
which finishes the proof.
\end{proof}

\begin{prop} \label{thcycl}
Let $1 \leq p < \infty$ and $f \in A^1_\beta(\T)$ with $\beta \geq 0$. We have
\begin{enumerate}
\item If $f$ is not cyclic in $A^p_\beta(\T)$ then there exists $S \in A^q_{-\beta}(\T) \setminus \{0\}$ such that $\supp(S) \subset \cZ(f)$.
\item If there exists a nonzero measure $\mu \in A^q_{-\beta}(\T)$ such that $\supp(\mu) \subset \cZ(f)$ then $f$ is not cyclic in $A^p_\beta(\T)$.
\end{enumerate}
\end{prop}

\begin{proof}
$(1)$ If $f$ is not cyclic in $A^p_\beta(\T)$, by duality there exists $S \in A^q_{-\beta}(\T) \setminus \{0\}$ such that
$$\langle S,e_nf \rangle=0, \qquad \forall n \in \Z.$$
Thus, by lemma \ref{supportZeros}, we have $\supp(S) \subset \cZ(f)$.

$(2)$ Let $\mu \in \A{q} \cap \MT \setminus \{0\}$ such that $\supp(\mu) \subset \cZ(f)$. Since $\mu$ is a measure on $\T$ we have $\langle \mu,e_nf \rangle=0$, for all $n \in \Z$. So $f$ is not cyclic in $A^p_\beta(\T)$.
\end{proof}

Recall that $A^1_\beta(\T)$ is a Banach algebra. Let $I$ be a closed ideal in $A^1_\beta(\T)$. We denote by $\cZ_I$ the set of common zeros of the functions of $I$,
$$\cZ_I = \bigcap_{f \in I} \cZ(f).$$
We have the following result about spectral synthesis in $A^1_\beta(\T)$.

\begin{lemme} \label{syntheseS} Let $0 \leq \beta < 1/2$. Let $I$ be a closed ideal in $A^1_\beta(\T)$. If $g$ is a Lipschitz function which vanishes on $\cZ_I$ then $g \in I$.
\end{lemme}

\begin{proof}
The proof is similar to the one given in \cite{KAH} pp. 121-123. For the sake of completeness we give the important steps. Let $I^\bot$ be the set of all $S$ in the dual space of $A^1_\beta(\T)$ satisfying $\langle S, f \rangle = 0$ for all $f \in I$. Let $g$ be a Lipschitz function which vanishes on $\cZ_I$ and $S \in I^\bot$. By Lemma \ref{supportZeros}, $\supp(S) \subset \cZ_I$. For $h > 0$, we set $S_h=S*\Delta_h$ where $\Delta_h : t \mapsto \frac{-|t|}{h^2}+\frac{1}{h}$ if $t \in [-h,h]$ and $0$ otherwise. We have $\widehat{\Delta_h}(0)=1/2\pi$ and $\widehat{\Delta_h}(n)=\frac{1}{2\pi} \frac{4\sin(nh/2)^2}{(nh)^2}$ for $n \neq 0$. Since $S$ is in the dual of $A^1_\beta(\T)$, $S_h \in A^1(\T)$. Moreover we have $\supp(S_h) \subset \supp(S) + \supp(\Delta_h) \subset \cZ_I^h := \cZ_I+[-h,h]$. We have
\begin{eqnarray*}
|\langle S_h , g \rangle|^2 & = & \left| \int_{\cZ_I^h\setminus \cZ(g)} S_h(x)g(x) dx \right|^2\\
& \leq & \left( \sum_{n \in \Z} |\widehat{S}(n)\widehat{\Delta_h}(n)|^2 \right) \left( \int_{\cZ_I^h\setminus \cZ(g)} |g(x)|^2 dx \right)\\
& \leq & C \left( \sum_{n \in \Z} \frac{\widehat{S}(n)^2 }{n^2}  \right) \left( |\cZ_I^h\setminus \cZ(g)| \right)
\end{eqnarray*}
where $C$ is a positive constant and where $|E|$ denotes the Lebesgue measure of $E$. So 
$\lim_{h \to 0} \langle S_h , g \rangle = 0.$
By the dominated convergence theorem, we obtain that
$$\lim_{h \to 0} \langle S_h , g \rangle = \lim_{h \to 0} \sum_{n \in \Z} \widehat{S_h}(n) \widehat{g}(-n) =  \frac{1}{2\pi} \sum_{n \in \Z} \widehat{S}(n)\widehat{g}(-n) = \frac{1}{2\pi} \langle S , g \rangle.$$
So $\langle S , g \rangle = 0$. Therefore $g \in I$.
\end{proof}

We also need the following lemma which is a consequence of Lemma \ref{syntheseS}. Newman gave a proof of this when $\beta = 0$ (see \cite[Lemma 2]{NEW}).

\begin{lemme} \label{lemme2Newman}
Let $0 \leq \beta < 1/2$ and a closed set  $E \subset \T$.\\
There exists $(f_n)$ a sequence of Lipschitz functions which are zero on $E$ and such that
$$\lim_{n \to \infty} \|f_n-1\|_{A^p_\beta(\T)} = 0$$
if and only if every $f \in A^1_\beta(\T)$ satisfying $\cZ(f)=E$ is cyclic in $A^p_\beta(\T)$.
\end{lemme}

\section{Proof of Theorem A}

Before proving Theorem A, let us recall Salem's Theorem (see \cite{SAL} and \cite{KAH} pp. 106-110).

\begin{theo} \label{Salem}
Let $0 < \alpha < 1$ and $q > \frac{2}{\alpha}$.\\
There exists a compact set $E \subset \T$ which satisfies $\dim(E)=\alpha$ and there exists a positive measure $\mu \in A^q(\T) \setminus \{0\}$ such that $\supp(\mu) \subset E$.
\end{theo}

To prove Theorem A, we also need the following result. The case $\beta = 0$ was considered by Newman in \cite{NEW}. For $k \in \N$ and $E \subset \T$, we denote 
$$k \times E = E + E + ... + E = \left\{ \sum_{n=1}^k x_n,~ x_n \in E \right\}.$$

\begin{theo} \label{ThNewman1}
Let $1 <  p < 2$ and $\beta > 0$ such that $\beta q \leq 1$, and let  $f \in A^1_\beta(\T)$.
\begin{enumerate}[(a)]
\item Let  $k \in \N^*$ be such that $k \leq q/2$.  If $C_\alpha(k \times \cZ(f)) = 0$ for some  $\alpha < \frac{2}{q} (1 - \beta q )k$, then $f$ is cyclic in $A^p_\beta(\T)$.
\item Let  $k \in \N^*$ be such that $ q/2\leq k \leq  1/(2\beta)$. If $C_\alpha(k \times \cZ(f)) = 0$ where  $\alpha = 1-2k\beta$, then $f$ is cyclic in $A^p_\beta(\T)$.
\end{enumerate}
\end{theo}

\begin{proof}
Let $k \in \N^*$. Suppose that $f$ is not cyclic in $A^p_\beta(\T)$. Then there exists $L \in A^q_{-\beta}(\T)$, the dual of $ A^p_\beta(\T)$, such that $L(1)=1$ and $L(Pf) = 0$, for all $P \in \PT$.

Since $\beta < \frac{1}{2}$, by \eqref{ineA1inf},  we get $C^1(\T)\subset A^1_\beta(\T)\subset A^p_\beta(\T)$, { and by \cite{NEW2}} (see also \cite[Lemma 5]{NEW}),  there exists $\phi \in L^2(\T)$ such that
$$L(g) = \int_\T \big(g'(x)\phi(x)+g(x)\big) \dd x, \qquad \forall g \in C^1(\T).$$
Since $L \in A^q_{-\beta}(\T)$ which implies $(L(e_n))_{n \in \Z} \in \ell^q_{-\beta}(\Z)$, we obtain
\begin{equation} \label{eqNewman}
\sum_{n \in \Z} |n \widehat{\phi}(n)|^q (1+|n|)^{-\beta q} < \infty.
\end{equation}
Moreover we have,
$$\int_\T\big( (e_nf)'(x)\phi(x)+(e_nf)(x) \big)\dd x = 0, \qquad n \in \Z,$$
and so $\langle \phi'-1,e_nf \rangle =0$ where $\phi'$ is defined in terms of distribution.
By \eqref{eqNewman}, $\phi'-1 \in A^q_{-\beta}(\T)$, so by lemma \ref{supportZeros}, we get $\supp(\phi'-1) \subset \cZ(f)$.

For $m \in \N$, we denote by $\phi^{*m}$ the result of convolving $\phi$ with itself $m$ times. Using the fact that $S'*T=S*T'$ and $1*S'=0$ for any distributions $S$ and $T$, we have
$$(\phi'-1) * \left(\left(\phi^{*(m-1)}\right)^{(m-1)}+(-1)^{m-1}\right) =  \left(\phi^{*m}\right)^{(m)} + (-1)^{m}.$$
So we can show by induction on $m \geq 1$ and by the formula $\supp(T * S) \subset \supp(T) + \supp(S)$ that
\begin{equation} \label{SupportPhi}
\supp \left(\left(\phi^{*m}\right)^{(m)} +(-1)^m \right) \subset m \times \cZ(f), \qquad \forall m \geq 1.
\end{equation}
Note that $\widehat{\left(\phi^{*k}\right)^{(k)}}(n)={i^k}n^k \widehat{\phi}(n)^k$ for $k \geq 1$ and $n \in \Z$.

\bigskip
$(a)$ : Suppose that $0 < k \leq q/2$ and $C_\alpha(k \times \cZ(f)) = 0$ for some  $\alpha < \frac{2}{q} (1 - \beta q )k$. We rewrite \eqref{eqNewman} as
$$\sum_{n \in \Z} \left( |n \widehat{\phi}(n)|^k \right)^{\frac{q}{k}} (1+|n|)^{-\frac{q}{k} \beta k} < \infty.$$
So, if we set $q' = \frac{q}{k} \geq 2$ and $\beta' = \beta k$, we have $\left(\phi^{*k}\right)^{(k)} \in A^{q'}_{-\beta'}(\T)$. By \eqref{SupportPhi} and by Lemma \ref{KahaneSupp} we obtain that  $\left(\phi^{*k}\right)^{(k)} = (-1)^{k-1}$. This contradicts the fact that   $\widehat{\left(\phi^{*k}\right)^{(k)}}(0)=0$.

\bigskip
$(b)$ : Now suppose that $ k \geq q/2$ and $C_\alpha(k \times \cZ(f)) = 0$ where $\alpha = 1-2k\beta$. Since $q \leq 2k$, we have by \eqref{eqNewman},
$$\sum_{n \in \Z} |n \widehat{\phi}(n)|^{2k} (1+|n|)^{-2k\beta} < \infty.$$
So $\left(\phi^{*k}\right)^{(k)} \in  A^2_{-k\beta}(\T)$ and $\left(\phi^{*k}\right)^{(k)} = (-1)^{k-1}$. Again this is absurd since $\widehat{\left(\phi^{*k}\right)^{(k)}}(0)=0$.
\end{proof}

We need to compute the capacity of the Minkowski sum of some Cantor type subset of $\T$. 
We denote by $[x]$ the integer part of $x \in \R$.
For $\lambda \in [0,1]$ and $k \in \N^*$, we define
$$K_\lambda^k = \{ m \in \N,~ \exists j \in \N, ~ m \in [2^j,2^j(1+\lambda+1/j)-k+1] \}$$
and we set in $\R/\Z \simeq [0,1[$,
$$S_\lambda^k = \left\{ x = \sum_{i=0}^\infty \frac{x_i}{2^{i+1}},~ (x_i) \in \{0,1\}^\N \text{ such that } i \in K_\lambda^k \Rightarrow x_i=0 \right\}.$$
We denote $K_\lambda=K_\lambda^1$ and $S_\lambda=S_\lambda^1$.

To prove $(4)$ of Theorem A we need the following lemma.

\begin{lemme} \label{dimHauss}
For all $k \geq 1$, we have 
\begin{enumerate}
\item $k \times S_\lambda \subset S_\lambda^k$ ;
\item $C_\alpha(S_\lambda^k)=0$ if and only if $\alpha \geq \frac{1-\lambda}{1+\lambda}$ ;
\item $\dim(k \times S_\lambda) = \frac{1-\lambda}{1+\lambda}$ and $C_{\frac{1-\lambda}{1+\lambda}}(k \times S_\lambda)=0$.
\end{enumerate}
\end{lemme}

\begin{proof}
$(1)$ : We prove this by induction. If $k=1$ we have $S_\lambda = S_\lambda^1$. We suppose the result true for $k-1$  for some $k\geq 2$ and we will show $k \times S_\lambda \subset S_\lambda^k$. We have $k \times S_\lambda \subset (k-1) \times S_\lambda + S_\lambda \subset S_\lambda^{k-1} + S_\lambda$. Let $x \in S_\lambda^{k-1}$, $y \in S_\lambda$ and $z=x+y$. Denote by $(x_i)$, $(y_i)$ and $(z_i)$ their binary decomposition. Let $m \in K_\lambda^k$. There exists $j \in \N$ such that $m \in [2^j,2^j(1+\lambda+1/j)-k+1]$. Since $m \in K_\lambda^k$, $m$ and $m+1$ are contained in $K_\lambda^{k-1} \subset K_\lambda$, we have $x_m=y_m=x_{m+1}=y_{m+1}=0$. Therefore we write
$$z=x+y = \sum_{i=0}^{m-1} \frac{x_i+y_i}{2^{i+1}} + \sum_{i=m+2}^\infty \frac{x_i+y_i}{2^{i+1}}.$$
Note that for infinitely many  $i \geq m+2$, $x_{i}+y_{i} < 2$, so we see that
$$\sum_{i=m+2}^\infty \frac{x_i+y_i}{2^{i+1}} < \frac{1}{2^{m+1}}.$$
Therefore, we obtain by uniqueness of the decomposition that $z_m=0$. This proves that $x+y \in S_\lambda^k$ and $k \times S_\lambda \subset S_\lambda^k$.

\bigskip
$(2)$ : We will study the capacity of $S_\lambda^k$ by decomposing it.
First we show that the set $S_\lambda^k$ is a generalized Cantor set in the sense of \cite{EID, OHT}. Let $\nu_j =[2^j(1+\lambda+1/j)-k+1]+1$ and $N_0$ (depending only on $k$ and $\lambda$) such that for all $j \geq N_0$, $2^j < \nu_j < 2^{j+1}$. We set for $N \geq N_0$,
$$l_N = \sum_{j=N}^\infty \frac{1}{2^{\nu_j}}-\frac{1}{2^{2^{j+1}}}.$$
Since $2^j(1+\lambda+1/j)-k+1 < \nu_j \leq 2^j(1+\lambda+1/j)-k+2$, we have
\begin{multline*}
\sum_{j=N}^\infty  \frac{1}{2^{2^j(1+\lambda+\frac{1}{j})}} \Big( \frac{1}{2^{2-k}} - \frac{1}{2^{2^j(1-\lambda-\frac{1}{j})}} \Big) \leq l_N \\ \leq \sum_{j=N}^\infty  \frac{1}{2^{2^j(1+\lambda+\frac{1}{j})}} \Big( \frac{1}{2^{1-k}} - \frac{1}{2^{2^j(1-\lambda-\frac{1}{j})}} \Big)
\end{multline*}
On one hand, there exists $C \geq 1$ such that for all  $j \geq N$, 
$$\frac{1}{C} \leq \frac{1}{2^{2-k}} - \frac{1}{2^{2^j(1-\lambda-\frac{1}{j})}} \leq \frac{1}{2^{1-k}} - \frac{1}{2^{2^j(1-\lambda-\frac{1}{j})}} \leq C.$$
On the other hand, for $N \geq N_0$,
\begin{eqnarray*}
\frac{1}{2^{2^N(1+\lambda+\frac{1}{N})}} \leq \sum_{j=N}^\infty  \frac{1}{2^{2^j(1+\lambda+\frac{1}{j})}} & \leq & \frac{1}{2^{2^N(1+\lambda+\frac{1}{N})}} + \sum_{j=0}^\infty \left(  \frac{1}{2^{2^{N+1}(1+\lambda)}} \right)^{2^j}\\
& \leq & \frac{1}{2^{2^N(1+\lambda+\frac{1}{N})}} + \sum_{j=0}^\infty \left(  \frac{1}{2^{2^{N+1}(1+\lambda)}} \right)^{j+1}\\
& \leq & \frac{1}{2^{2^N(1+\lambda+\frac{1}{N})}} +  \frac{2}{2^{2^{N+1}(1+\lambda)}}\\
& \leq & \frac{3}{2^{2^N(1+\lambda+\frac{1}{N})}}.
\end{eqnarray*}
Hence we obtain that $l_N$ is comparable to $2^{-2^N(1+\lambda+1/N)}$, that is:
\begin{equation} \label{lNCompare}
\frac{1}{C 2^{2^N(1+\lambda+\frac{1}{N})}} \leq l_N  \leq  \frac{3C}{2^{2^N(1+\lambda+\frac{1}{N})}}.
\end{equation}
Moreover we have
\begin{equation} \label{lNestimation}
l_N = \frac{1}{2^{\nu_N}} - \sum_{j=N+1}^\infty \frac{1}{2^{2^{j}}} - \frac{1}{2^{\nu_j}} < \frac{1}{2^{\nu_N}} \leq \frac{1}{2^{2^N}}.
\end{equation}
We set 
$$E_N = \left\{\sum_{i=0}^{2^N-1} \frac{x_i}{2^{i+1}} + l_N z,~ z \in [0,1],~ x_i \in \{0,1\},~ i \in K_\lambda^k \Rightarrow x_i=0 \right\}.$$
We can see $E_N$ as a union of disjoint intervals by writing
$$E_N = \bigcup_{\substack{(x_i) \in \{0,1\}^{2^N} \\ i \in K_\lambda^k \Rightarrow x_i=0}} E_N^{(x_i)},$$
where $$E_N^{(x_i)} = \sum_{i=0}^{2^N-1} \frac{x_i}{2^{i+1}} + l_N [0,1].$$
Note that the intervals $E_N^{(x_i)}$ are disjoint since by \eqref{lNestimation}, $l_N < \frac{1}{2^{2^N}}$.
For fixed $N \geq N_0$, let $(x_i)_{0 \leq i \leq 2^N-1} \in \{0,1\}^{2^N}$ and $(y_i)_{0 \leq i \leq 2^{N+1}-1} \in \{0,1\}^{2^{N+1}}$.\\
Claim : $E_{N+1}^{(y_i)} \subset E_{N}^{(x_i)}$ if and only if $x_i=y_i$ for all $0 \leq i < 2^N$ and $y_i=0$ for all $2^N \leq i < \nu_N$.\\
Indeed, suppose that $E_{N+1}^{(y_i)} \subset E_{N}^{(x_i)}$ and let $u \in E_{N+1}^{(y_i)}$. We have
$$u = \sum_{i=0}^{2^N-1} \frac{x_i}{2^{i+1}} + l_N z_1 = \sum_{i=0}^{2^{N+1}-1} \frac{y_i}{2^{i+1}} + l_{N+1} z_2,$$
where $z_1$ and $z_2$ are in $[0,1]$. By \eqref{lNestimation}, $l_N < \frac{1}{2^{\nu_N}}$, and using the uniqueness of the binary representation, we obtain $x_i=y_i$ for all $0 \leq i < 2^N$ and $y_i=0$ for all $2^N \leq i < \nu_N$.\\
Now suppose $x_i=y_i$ for all $0 \leq i < 2^N$ and $y_i=0$ for all $2^N \leq i < \nu_N$. Let $u \in E_{N+1}^{(y_i)}$. We write
$$u = \sum_{i=0}^{2^N-1} \frac{x_i}{2^{i+1}} +  \sum_{i=\nu_N}^{2^{N+1}-1} \frac{y_i}{2^{i+1}} + l_{N+1} z,$$
where $z \in [0,1]$. Note that 
\begin{equation}\label{equidistant}
 \sum_{i=\nu_N}^{2^{N+1}-1} \frac{1}{2^{i+1}} + l_{N+1} = \frac{1}{2^{\nu_N}}-\frac{1}{2^{2^{N+1}}} + l_{N+1} = l_N.
\end{equation}
So we have
$$ \sum_{i=0}^{2^N-1} \frac{x_i}{2^{i+1}} \leq  \sum_{i=0}^{2^N-1} \frac{x_i}{2^{i+1}} +  \sum_{i=Z_N}^{2^{N+1}-1} \frac{y_i}{2^{i+1}} + l_{N+1} z \leq \frac{1}{2} \sum_{i=0}^{2^N-1} \frac{x_i}{2^i} + l_N,$$
and $u \in E_{N}^{(x_i)}$. This conclude the proof of the claim.

\bigskip
By the claim, for fixed $(x_i)$ and  for $N \geq N_0$, we have the following properties :
\begin{enumerate}[(i)]
\item the interval $E_{N}^{(x_i)}$ contains precisely 
$$p_N = \# \{ (y_i)_{\nu_N \leq i \leq 2^{N+1}-1}\text{: } y_i \in \{0,1\} \} = 2^{2^{N+1}-\nu_N}$$ intervals of the form $E_{N+1}^{(y_i)},$
\item the intervals of the form $E_{N+1}^{(y_i)}$ contained in $E_{N}^{(x_i)}$ are equidistant intervals of length $l_{N+1}$:  the distance of two consecutive intervals of the form $E_{N+1}^{(y_i)}$  is equal to  $\frac{1}{2^{2^{N+1}}-l_{N+1}}$, 
\item if we denote $E_{N}^{(x_i)}=[a,b]$ then there exist $(y_{i})$ and $(z_{i})$ such that $E_{N+1}^{(y_i)}=[a,a+l_{N+1}]$ and $E_{N+1}^{(z_i)}=[b-l_{N+1},b]$.
\end{enumerate}
Finally we can write $S_\lambda^k$ as
$$S_\lambda^k=\bigcap_{  N \geq N_0} E_N.$$
This shows that $S_\lambda^k$ is a generalized Cantor set in the sense of \cite{EID,OHT}. So, by \cite{EID,OHT}, we have for $0<\alpha<1$ that $C_\alpha(S_\lambda^k)=0$ if and only if
$$\sum_{N=N_0}^\infty \frac{1}{(p_{N_0} \cdots p_{N-1}) l_N^\alpha} = \infty.$$
Since
\begin{multline*}
2^{(k-2)(N-N_0) + (2^N-2^{N_0})(1-\lambda) - \sigma_N} \leq p_{N_0} \cdots p_{N-1} \\ \leq  2^{(k-1)(N-N_0) + (2^N-2^{N_0})(1-\lambda) - \sigma_N},
\end{multline*}
where
$$\sigma_N = \sum_{j=N_0}^{N-1} \frac{2^j}{j},$$
we have, by \eqref{lNCompare}, $C_\alpha(S_\lambda^k)=0$ if and only if
$$\sum_{N=N_0}^\infty 2^{2^N (\alpha (1+\lambda) - (1-\lambda)) + \alpha 2^N/N + \sigma_N - (k-1)(N-N_0) + 2^{N_0}(1-\lambda)} = \infty.$$
Therefore $C_\alpha(S_\lambda^k)=0$ if and only if $\alpha \geq \frac{1-\lambda}{1+\lambda}$. 

\bigskip
$(3)$ immediately follows from $(1)$ and $(2)$ by the capacity property \eqref{PropCap}.
\end{proof}

We are ready to prove Theorem A. The following Theorem is a reformulation of Theorem A in $A^p_\beta(\T)$ spaces.

\begin{theo} \label{ThAA}
Let $1 < p < 2$, $\beta > 0$ such that $\beta q \leq 1$.
\begin{enumerate}
\item If  $f \in A^1_\beta(\T)$ and $\dim(\cZ(f)) < \frac{2}{q}(1 - \beta q)$ then $f$ is cyclic in $A^p_\beta(\T)$.
\item If  $f \in A^1_\beta(\T)$ and $C_{1-\beta q}(\cZ(f)) > 0$ then $f$ is not cyclic in $A^p_\beta(\T)$.
\item For $\frac{2}{q}(1 - \beta q) < \alpha \leq 1$, there exists a closed set $E \subset \T$ such that $\dim(E)=\alpha$ and every $f \in A^1_\beta(\T)$ satisfying $\cZ(f)=E$ is not cyclic in $A^p_\beta(\T)$.
\item Let $k=[q/2]$. For all $\varepsilon > 0$, there exists a closed set $E\subset \T$ such that 
$$\dim(E) \geq \max \left(  \frac{2}{q} (1 - \beta q )k - \varepsilon,~ 1-2(k+1)\beta \right)$$
and such that every $f \in A^1_\beta(\T)$ satisfying $\cZ(f) = E$ is cyclic in $A^p_\beta(\T)$.\\
Furthermore, if $p = \frac{2k}{2k-1}$ for some $k \in \N^*$, $E$ can be chosen such that $\dim(E)=1-\beta q$.
\end{enumerate}
\end{theo}

\begin{proof}
$(1)$ : Note that, by \eqref{PropCap}, $\dim(\cZ(f)) < \frac{2}{q}(1 - \beta q)$ if and only if there exists $\alpha < \frac{2}{q}(1 - \beta q)$ such that $C_\alpha(\cZ(f)) = 0$. If $C_\alpha(\cZ(f)) = 0$, by Lemma \ref{KahaneSupp}, there is no $S \in A^q_{-\beta}(\T) \setminus \{0\}$ such that $\supp(S) \subset \cZ(f)$ . So, by Proposition \ref{thcycl} (1), $f$ is cyclic in $A^p_\beta(\T)$.

\bigskip \noindent
$(2)$ : Suppose that $C_{1-\beta q}(\cZ(f)) > 0$. There exists a probability measure $\mu$ of energy $I_{1-\beta q}(\mu) < \infty$, such that $\supp(\mu) \subset \cZ(f)$ . So $\mu \in A^2_{-\beta q/2}(\T) \setminus \{ 0 \}$. Since $|\widehat{\mu}(n)| \leq 1$ for all $n \in \Z$
and $q \geq 2$, we have $\mu \in A^q_{-\beta}(\T)$. By proposition \ref{thcycl} (2), $f$ is not cyclic in $A^p_\beta(\T)$.

\bigskip \noindent
$(3)$ : Let $\frac{2}{q}(1 - \beta q) < \alpha \leq 1$. There exists $\varepsilon >0$ such that $\frac{2}{q}(1 - \beta q) + \varepsilon < \alpha$. Let $q'$ such that $\frac{2}{q}-2\beta+\varepsilon = \frac{2}{q'}$. Since $\beta > \frac{1}{q}-\frac{1}{q'}$, by Lemma \ref{inclusionAp},  $A^{q'}(\T) \subset A^q_{-\beta}(\T)$. 
By Theorem \ref{Salem}, as $q'$ satisfies $q' > \frac{2}{\alpha}$, there exists a closed subset $E \subset \T$ such that $\dim(E)=\alpha$ and a non zero positive measure $\mu \in A^{q'}(\T) \subset A^q_{-\beta}(\T)$ such that $\supp(\mu) \subset E$. Now $(3)$ follows from proposition \ref{thcycl}.$(2)$.

\bigskip \noindent
$(4)$ : Let $k=[q/2]$. Suppose first $\frac{2}{q} (1 - \beta q )k > 1-2(k+1)\beta$ and let $0<\varepsilon'<\varepsilon$ satisfying $1-2(k+1)\beta \leq \frac{2}{q} (1 - \beta q )k - \varepsilon'$. Consider the set $S_\lambda$ where $\lambda$ verifies $$\frac{2}{q} (1 - \beta q )k - \varepsilon' < \frac{1-\lambda}{1+\lambda} < \frac{2}{q} (1 - \beta q )k.$$
By Lemma \ref{dimHauss}.(3) we have $\dim(S_\lambda) = \frac{1-\lambda}{1+\lambda}$ and $C_{\frac{1-\lambda}{1+\lambda}}(k \times S_\lambda)=0$. Therefore by Theorem \ref{ThNewman1}.$(a)$, every $f \in A^1_\beta(\T)$ such that $\cZ(f) = S_\lambda$ is cyclic in $A^p_\beta(\T)$.

Now suppose $\frac{2}{q} (1 - \beta q )k \leq 1-2(k+1)\beta$. We consider $S_\lambda$ where $\frac{1-\lambda}{1+\lambda} = 1-2(k+1)\beta$. By lemma \ref{dimHauss}.(3) we have $\dim(S_\lambda)=\frac{1-\lambda}{1+\lambda}=1-2(k+1)\beta$ and $C_{\frac{1-\lambda}{1+\lambda}}((k+1) \times S_\lambda)=0$. So by Theorem \ref{ThNewman1}.$(b)$, every $f \in A^1_\beta(\T)$ such that $\cZ(f) = S_\lambda$ is cyclic in $A^p_\beta(\T)$.

\bigskip
Suppose now that $p = \frac{2k}{2k-1}$ for some $k \in \N^*$. As before, we consider $S_\lambda$ where $\frac{1-\lambda}{1+\lambda} = 1-2k\beta = 1-\beta q$. So again by Theorem \ref{ThNewman1}.$(b)$, every $f \in A^1_\beta(\T)$ such that $\cZ(f) = S_\lambda$ is cyclic in $A^p_\beta(\T)$.
\end{proof}

Note that the set $E$ which is considered in \ref{ThAA}.(4) verifies $C_{\alpha}(E)=0$ where 
$$\alpha \geq \max \left(  \frac{2}{q} (1 - \beta q )k - \varepsilon,~ 1-2(k+1)\beta \right).$$

\section{Proof of Theorem B}

In this section we investigate the sharpness of the constant $\frac{2}{q} (1 - \beta q )$ in Theorem A.

Before proving Theorem B, we need the following two results.
The following Lemma is an extension of Newman's Lemma 3 (see \cite{NEW} pp 654-655).
\begin{lemme} \label{lemme3Newman}
Let $p \in [1,2[$, $\beta \geq 0$ such that $\beta q \leq 1$. There exists $C > 0$ such that 
for all $f \in A^2_1(\T)$,
$$\|f\|_{A^p_\beta(\T)} \leq C^{\frac{1}{p}} \|f\|_{A^2(\T)}^{\frac{3}{2}-\frac{1}{p}-\beta} (\|f\|_{A^2(\T)} + \|f'\|_{A^2(\T)})^{\frac{1}{p}-\frac{1}{2}+\beta}.$$
\end{lemme}

\begin{proof}
It suffices to show that there exists $C > 0$ such that for all sequences $(c_n) \in \C^{\N^*}$,
$$\sum_{n = 1}^\infty |c_n|^p (1+|n|)^{p\beta} \leq C \left( \sum_{n = 1}^\infty |c_n|^2 \right)^{\frac{3}{4}p-\frac{1}{2}-\frac{\beta p}{2}} \left( \sum_{n = 1}^\infty n^2|c_n|^2 \right)^{\frac{1}{2}-\frac{p}{4}+\frac{\beta p}{2}}.$$
Then we apply this inequality to $(\widehat{f}(n))_{n \geq 1}$ and $(\widehat{f}(-n))_{n \geq 1}$.
Let $x^2 = \sum_{n \geq 1} |c_n|^2$ and $x^2y^2 = \sum_{n \geq 1} n^2|c_n|^2$. Note that $y \geq 1$. On one hand, by the H\"older inequality,
\begin{eqnarray*}
\sum_{1 \leq n \leq y} |c_n|^p (1+n)^{p\beta} & \leq & \left( \sum_{n=1}^y |c_n|^2 \right)^{p/2} \left( \sum_{n=1}^y (1+n)^{\frac{2p\beta}{2-p}} \right)^{1-p/2}\\
& \leq &  \left( x^2 \right)^{\frac{p}{2}} \left( y (1+y)^{\frac{2p\beta}{2-p}} \right)^{1-\frac{p}{2}} \leq 2^{\beta p} x^p y^{1-\frac{p}{2}+p\beta}.
\end{eqnarray*}
On the other hand we set $\gamma = \frac{2p(1-\beta)}{2-p}$. Since $\beta q \leq 1$, $\gamma > 1$ and again by the H\"older inequality we obtain,
\begin{eqnarray*}
\sum_{n > y} |c_n|^p (1+n)^{p\beta} &\leq & 2^p \left( \sum_{n > y}^\infty n^2|c_n|^2 \right)^{\frac{p}{2}} \left( \sum_{n > y}^\infty (1+n)^{\frac{2p(\beta-1)}{2-p}} \right)^{1-\frac{p}{2}} \\
& \leq & 2^{ p} \left( x^2y^2 \right)^{\frac{p}{2}} ~ \left( \frac{1}{\gamma-1} \right)^{1-p/2} \left( y^{1-\gamma} \right)^{1-\frac{p}{2}} \\
&\leq & 2^{p} \left( \frac{1}{\gamma-1} \right)^{1-p/2} x^p y^{1-\frac{p}{2}+p\beta}
\end{eqnarray*}
So the conclusion of the Lemma holds with $$C = \max\left( 2^{\beta p}, 2^{p} \left( \frac{1}{\gamma-1} \right)^{1-p/2}\right)$$ which is a positive constant depending only on p and $\beta$.
\end{proof}

The following theorem is due to K\"orner (see \cite[Theorem 1.2]{KOR}).

\begin{theo} \label{Korner}
Let $h : [0,\infty) \rightarrow [0,\infty)$ be an increasing continuous function with $h(0)=0$ and let $\phi : [0,\infty) \rightarrow [0,\infty)$ be a decreasing function.
Suppose that
\begin{enumerate}
\item $\int_1^\infty \phi(x)^2 dx = \infty$;
\item there exist $K_1,K_2 > 1$ such that for all $1 \leq x \leq y \leq 2x$, $$K_1 \phi(2x) \leq \phi(x) \leq K_2 \phi(y) ;$$
\item there exists $\gamma > 0$ such that
$$\lim_{x \to \infty} x^{1-\gamma} \phi(x) = \infty ;$$
\item there exist $0 < K_2 < K_3 < 1$ such that for all $t > 0$, 
$$K_2 h(2t) \leq h(t) \leq K_3 h(2t).$$
\end{enumerate}
Then there exists a probability measure $\mu$ with support of Hausdorff $h$-measure zero such that 
$$|\widehat{\mu}(n)| \leq \phi \left(\frac{1}{h(|n|^{-1})} \right) \left(\ln \left(\frac{1}{h(|n|^{-1})} \right)\right)^{1/2}, \qquad \forall n \neq 0.$$
\end{theo}

Recall Theorem B reformulated in $A^p_\beta(\T)$ space.

\begin{theo}
Let $1 < p < 2$ and $\beta \geq 0$ such that $\beta q < 1$.
\begin{enumerate}
\item If $f \in A^1_\beta(\T)$ and $\cZ(f)$ has strong $\alpha$-measure $0$ where $\alpha = \frac{2}{q}(1 - \beta q)$ then $f$ is cyclic in $A^p_\beta(\T)$.
\item For every $\gamma > \frac{2}{q}$, there exists a closed subset $E \subset \T$ such that every $f \in A^1_\beta(\T)$ satisfying $\cZ(f)=E$ is not cyclic in $A^p_\beta(\T)$ and such that $H_h(E)=0$ where $h(t)=\frac{t^\alpha}{\ln(e/t)^\gamma}$ with $\alpha = \frac{2}{q}(1 - \beta q)$.
\end{enumerate}
\end{theo}

Note that in $(2)$, $H_h$ is closed to $H_\alpha$.

\begin{proof}
$(1)$ : The proof of this result holds by using arguments analogous to those of Newman for $\beta = 0$ (see \cite[Theorem 1]{NEW}).
Denote by $(a_k,b_k)$ the complementary intervals of $\cZ(f)$ arranged in non-increasing order of lengths and set
$$r_n = 2\pi - \sum_{k=0}^n (b_k-a_k).$$
The set $\cZ(f)$ has strong $\alpha$-measure $0$ where $\alpha = \frac{2}{q}(1 - \beta q)$ so $$\lim_{n \to \infty} ~ r_n ~ n^{\frac{1}{\alpha}-1} = 0.$$
Let $\varepsilon>0$ and $n\in \N$ such that $r_n < \varepsilon n^{1-\frac{1}{\alpha}}$ and $\varepsilon n^{-\frac{1}{\alpha}}<1$.
Let the function  $\psi$ be given by 
$$\psi(x) = \max \Big( 1 - \frac{n^{\frac{1}{\alpha}}}{\varepsilon} \rho(x) , ~0 \Big),\qquad x\in \T,$$
where 
$$\rho(x)=\dist\big(x,\T \setminus \bigcup_{k=1}^n ]a_k,b_k[\big).$$
Then
\begin{multline*}
\|\psi\|_{A^2(\T)}^2 = \int_{\T \setminus \bigcup_{k=1}^n ]a_k,b_k[} \psi(t)^2 \dd t + \sum_{k=1}^n \int_{a_k}^{b_k} \psi(t)^2 \chi_{\{\rho(x) \leq \varepsilon n^{-\frac{1}{\alpha}} \}}(t) \dd t\\
 \leq  r_n + \sum_{k=1}^n 2 \varepsilon n^{-\frac{1}{\alpha}}
 \leq  3 \varepsilon n^{1-\frac{1}{\alpha}}.
\end{multline*}
Moreover
\begin{multline*}
\|\psi'\|_{A^2(\T)}^2 = \int_\T \psi'(t)^2 \dd t
= \sum_{k=1}^n \int_{a_k}^{b_k} \psi'(t)^2 \chi_{\{\rho(x) \leq \varepsilon n^{-\frac{1}{\alpha}} \}}(t) \dd t\\
 \leq  \sum_{k=1}^n \left( \frac{n^{\frac{1}{\alpha}}}{\varepsilon} \right)^2 2 \varepsilon n^{-\frac{1}{\alpha}}
 \leq  2 \frac{n^{1+\frac{1}{\alpha}}}{\varepsilon}.
\end{multline*}
Since $ \varepsilon n^{-\frac{1}{\alpha}} < \frac{n^{\frac{1}{\alpha}}}{\varepsilon}$ and $\alpha = \frac{2}{q}(1 - \beta q)$, by Lemma \ref{lemme3Newman}, 
$$\|\psi\|_{A^p_\beta(\T)} \leq C^{\frac{1}{p}} \left(3 \varepsilon n^{1-\frac{1}{\alpha}} \right)^{\frac{3}{4}-\frac{1}{2p}-\frac{\beta}{2}} \left(5 \frac{n^{1+\frac{1}{\alpha}}}{\varepsilon} \right)^{\frac{1}{2p}-\frac{1}{4}+\frac{\beta}{2}} \leq C' \varepsilon^{1-\frac{1}{p}-\beta}$$
where $C$ and $C'$ depend only on $\beta$ and $p$.  Note that $1-\psi$ is a Lipschitz function and  $\cZ(f) \subset \cZ(1-\psi)$. We conclude by Lemma \ref{lemme2Newman}.

\bigskip \noindent
$(2)$ : Let $\alpha = \frac{2}{q}(1 - \beta q)$ and $\gamma > \frac{2}{q}$. By Theorem \ref{Korner} with $\phi(t) = (t \ln(et))^{-1/2}$ for $t \geq 1$ and $ h(t) = \frac{t^\alpha}{\ln(e/t)^\gamma}$ for $t \in [0,\infty)$, there exists a probability measure $\mu$ with support of Hausdorff $h$-measure zero such that 
$$|\widehat{\mu}(n)| \leq \phi \left(\frac{1}{h(|n|^{-1})} \right) \left(\ln \left(\frac{1}{h(|n|^{-1})} \right)\right)^{1/2} \leq (|n|^{\alpha} \ln(e|n|)^{\gamma})^{-1/2},$$
for $n \neq 0$. So
\begin{eqnarray*}
\sum_{n \neq 0} | \widehat{\mu}(n) |^q (1+|n|)^{-\beta q} & \leq & C \sum_{n \neq 0} | n |^{-\alpha q/2 - \beta q} \ln(e|n|)^{-\gamma q/2}\\
& \leq & C \sum_{n \neq 0} \frac{1}{|n| \ln(e|n|)^{\gamma q/2}} < \infty
\end{eqnarray*}
with $C$ a positive constant. Hence $\mu \in A^q_{-\beta}(\T)$. We set $E = \supp(\mu)$. By lemma \ref{thcycl} the result is proved.

\end{proof}

\section{Remarks}

We say that $(\omega_n) \in \R^\Z$ is a weight if $w_n \geq 1$ and $\omega_{n+k} \leq C \omega_n \omega_k$ for all $k,n \in \Z$ and $C$ a positive constant. For $\omega$ a weight and $1 \leq p < \infty$ we set $$A^p_\omega(\T) = \left\{ f \in C(\T), \;\; 
 \|f\|_{A^p_\omega(\T)}^p = \sum_{n\in \Z} |\widehat{f}(n)|^p \omega_n^{p} < \infty \right\}.$$
Note that $\|fS\|_{A^p_\omega(\T)} \leq \|f\|_{A^1_\omega(\T)} \|S\|_{A^p_\omega(\T)}$ for $f \in A^1_\omega(\T)$ and $S \in A^p_\omega(\T)$. So we have the same result as \eqref{caraCyclNorm} to characterize cyclicity in $A^p_\omega(\T)$ by norm.

When $\omega_n=O((1+|n|)^\varepsilon)$ for all $\varepsilon > 0$, for example $\omega_n=\ln(e+|n|)^\beta$ where $\beta \geq 0$, we can show the same result as Lemma \ref{lemme2Newman}. So by noting that for all $p \geq 1$ and $\delta > 0$,
$$A^p_\delta(\T) \subset A^p_\omega(\T) \subset A^p(\T)$$
we obtain by Theorem A the following result:

\begin{theo} \label{Thlog}
Let $1 < p < 2$ and $\omega=(\omega_n)_{n\in\Z}$ a weight satisfying $\omega_n=O((1+|n|)^\varepsilon)$ for all $\varepsilon > 0$.
\begin{enumerate}
\item If $f \in A^1_\omega(\T)$ and $\dim(\cZ(f)) < \frac{2}{q}$ then $f$ is cyclic in $A^p_\omega(\T)$.
\item For $\frac{2}{q} < \alpha \leq 1$, there exists a closed subset $E \subset \T$ such that $\dim(E)=\alpha$ and every $f \in A^1_\omega(\T)$ satisfying $\cZ(f)=E$ is not cyclic in $A^p_\omega(\T)$.
\item For all $0 < \varepsilon < 1$, there exists a closed subset $E \subset \T$ such that $\dim(E)=1-\varepsilon$ and every $f \in A^1_\omega(\T)$ satisfying $\cZ(f)=E$ is cyclic in $A^p_\omega(\T)$.
\end{enumerate}
\end{theo}

\begin{proof}
$(1)$ : Let $f \in A^1_\omega(\T)$ such that $\dim(\cZ(f)) < \frac{2}{q}$. There exists $0 < \delta < 1/2$ such that $\dim(\cZ(f)) < \frac{2}{q}(1 - \delta q)$. By Theorem \ref{ThAA}.(1), every $g \in A^1_\delta(\T)$ satisfying $\cZ(g)=\cZ(f)$ is cyclic in $A^p_\delta(\T)$. Therefore by Lemma \ref{lemme2Newman}, there exist $(f_n)$ a sequence of Lipschitz functions which are zero on $\cZ(f)$ and such that
$$\lim_{n \to \infty} \|f_n-1\|_{A^p_\delta(\T)} = 0.$$
Moreover $\omega_n=O((1+|n|)^\delta)$ so $$\lim_{n \to \infty} \|f_n-1\|_{A^p_\omega(\T)} = 0.$$
Again by Lemma \ref{lemme2Newman} in $A^p_\omega(\T)$, we obtain that $f$ is cyclic in $A^p_\omega(\T)$.

$(2)$ : By the theorem of Salem (see Theorem \ref{Salem} and Theorem \ref{ThSansPoids}.(2)), there exists a closed set $E \subset \T$ such that $\dim(E)=\alpha$ and every $f \in A^1(\T)$ satisfying $\cZ(f)=E$ is not cyclic in $A^p(\T)$. Let $f \in A^1_\omega(\T)$ such that $\cZ(f)=E$ . Since $f \in A^1(\T)$, $f$ is not cyclic in $A^p(\T)$. However $\| \cdot \|_{A^p(\T)} \leq \| \cdot \|_{A^p_\omega(\T)}$ therefore $f$ is not cyclic in $A^p_\omega(\T)$.

$(3)$ : Let $0 < \varepsilon < 1$ and $\beta > 0$ such that $1-2([q/2]+1)\beta \geq 1-\varepsilon$. By Theorem \ref{ThAA}.(4), there exists a closed set $E\subset \T$ such that 
$$\dim(E) \geq  1-2([q/2]+1)\beta \geq 1-\varepsilon$$
and such that every $f \in A^1_\beta(\T)$ satisfying $\cZ(f) = E$ is cyclic in $A^p_\beta(\T)$. Since $A^p_\beta(\T) \subset A^p_\omega(\T)$, we obtain, by Lemma \ref{lemme2Newman}, that every $f \in A^1_\omega(\T)$ satisfying $\cZ(f)=E$ is cyclic in $A^p_\omega(\T)$.
\end{proof}

\bigskip
When $p > 2$ the search for cyclic vectors in $A^p(\T)$ seems extremely difficult. Newman in \cite{NEW} shows that for all $\alpha < 2\pi$ there exists $E \subset \T$ which has a Lebesgue measure $|E| > \alpha$ and such that every $f \in A^1(\T)$ satisfying $\cZ(f)=E$ is cyclic in $A^p(\T)$. See also \cite[Theorem 6]{NEW} for the existence of non cyclic functions under some conditions. We also have a characterization of the cyclic vectors in term of the zeros of the Fourier transform when $p > 2$ but it's not very effective : A function $f \in A^1(\T)$ is cyclic in $A^p(\T)$ if and only if $\cZ(f)$ does not support any non-zero function $g \in A^q(\T)$ where $q = \frac{p}{p-1}$.

When $\omega_n=\log(e+|n|)^\beta$ where $0< \beta < 1$, for all $p > \frac{2}{1-\beta}$ and for all $\alpha < 2\pi$, Nikolskii shows in \cite[Corollary 6]{NIK}, there exists $E \subset \T$ which has a Lebesgue measure $|E| > \alpha$ and such that every $f \in A^1_\beta(\T)$ satisfying $\cZ(f) = E$ is cyclic in $A^p_\beta(\T)$.

\section*{Acknowledgements}

I would like to acknowledge my doctoral advisors, K. Kellay and M. Zarrabi.

\end{document}